\documentclass{amsart}

\usepackage{amsmath}
\usepackage{paralist}
\usepackage{graphics}
\usepackage{epsfig}
\usepackage{graphicx}
\usepackage{epstopdf}
\usepackage[colorlinks=true]{hyperref}
\hypersetup{urlcolor=blue, citecolor=red}

\usepackage{psfrag}
\usepackage{subfigure}

\newtheorem{theorem}{Theorem}[section]
\newtheorem{corollary}{Corollary}
\newtheorem{lemma}[theorem]{Lemma}
\theoremstyle{definition}
\newtheorem{definition}[theorem]{Definition}
\newtheorem{remark}{Remark}

\def\a{{\alpha(t)}}
\def\t{\tau}
\def\LHI{{_a\mathcal{I}_t^{\a}}}
\def\RHI{{_t\mathcal{I}_b^{\a}}}
\def\LHD{{_a\mathcal{D}_t^{\a}}}
\def\RHD{{_t\mathcal{D}_b^{\a}}}
\def\LHMD{{_a\mathbb{D}_t^{\a}}}
\def\RHMD{{_t\mathbb{D}_b^{\a}}}
\def\C{\left(^{-\a}_{\quad k}\right)}
\def\D{\left(^{\,\,-\a}_{k-n-1}\right)}
\def\DS{\displaystyle}

\title[Computing Hadamard type operators]{Computing Hadamard type operators\\ 
of variable fractional order\footnote{\normalfont\scshape\lowercase{This is 
a preprint of a paper whose final and definite form will be published in 
Applied Mathematics and Computation, ISSN: 0096-3003. 
Submitted 10/June/2014; revised 03/Dec/2014; accepted 16/Dec/2014.}}}

\author[R. Almeida and D. F. M. Torres]{}



\email{ricardo.almeida@ua.pt}
\email{delfim@ua.pt}

\begin{document}

\maketitle


\centerline{\scshape Ricardo Almeida and Delfim F. M. Torres}
\medskip
{\footnotesize
 \centerline{Center for Research and Development in Mathematics and Applications (CIDMA)}
   \centerline{Department of Mathematics, University of Aveiro}
   \centerline{3810--193 Aveiro, Portugal}
}

\bigskip


\begin{abstract}
We consider Hadamard fractional derivatives and integrals of variable fractional order.
A new type of fractional operator, which we call the Hadamard--Marchaud fractional derivative,
is also considered. The objective is to represent these operators as series of terms
involving integer-order derivatives only, and then approximate the fractional operators
by a finite sum. An upper bound formula for the error is provided. We exemplify our method
by applying the proposed numerical procedure to the solution of a fractional differential equation
and a fractional variational problem with dependence on the Hadamard--Marchaud fractional derivative.
\end{abstract}


\section{Introduction}

Fractional calculus is a branch of mathematical analysis that studies the possibility
of taking real number powers or complex number powers of the differentiation
and integration operators \cite{Samko:book,Ref2:1}.
It has been called ``The calculus of the XXI century''
(K. Nishimoto, 1989) and claimed that ``Nature works with fractional time derivatives''
(S. Westerlund, 1991) \cite{Uchaikin}. Several definitions for fractional derivatives
and fractional integrals are found in the literature. Although the most common ones seem to be
the Riemann--Liouville and Caputo fractional operators, recently there has been an increasing interest
in the development of Hadamard's XIX century fractional calculus \cite{Hadamard}: see
\cite{MyID:290,MyID:292,MR1907120,Butzer,Butzer2,Jarad,Katugampola,Kilbas2,Kilbas,Kilbas3} and references therein.
This calculus is due to the French mathematician Jacques Hadamard (1865--1963),
where instead of power functions, as in Riemann--Liouville and Caputo fractional calculi,
one has logarithm functions. The left and right Hadamard fractional integrals
of order $\alpha>0$ are defined by
$$
{_a\mathcal{I}_t^\alpha}x(t)=\frac{1}{\Gamma(\alpha)}
\int_a^t\left(\ln\frac{t}{\tau}\right)^{\alpha-1}\frac{x(\tau)}{\tau}d\tau
$$
and
$$
{_t\mathcal{I}_b^\alpha}x(t)=\frac{1}{\Gamma(\alpha)}
\int_t^b \left(\ln\frac{\tau}{t}\right)^{\alpha-1}\frac{x(\tau)}{\tau}d\tau,
$$
respectively, while the left and right Hadamard fractional derivatives
of order $\alpha\in(0,1)$ are given by
\begin{equation}
\label{classicalHD}
{_a\mathcal{D}_t^\alpha}x(t)=\frac{t}{\Gamma(1-\alpha)}
\frac{d}{dt}\int_a^t \left(\ln\frac{t}{\tau}\right)^{-\alpha}\frac{x(\tau)}{\tau}d\tau
\end{equation}
and
$$
{_t\mathcal{D}_b^\alpha}x(t)=\frac{-t}{\Gamma(1-\alpha)}\frac{d}{dt}
\int_t^b \left(\ln\frac{\tau}{t}\right)^{-\alpha}\frac{x(\tau)}{\tau}d\tau,
$$
respectively. Uniqueness and continuous dependence of solutions for
nonlinear fractional differential systems with Hadamard
derivatives is discussed in \cite{MR3215482}.
The main purpose of this paper is to extend the previous definitions to the case where
the order $\alpha$ of the integrals and of the derivatives is not a constant, but a function
that depends on time. Such time-dependence of $\alpha$ has already been considered
for Riemann--Liouville and Caputo fractional operators, and has proven to describe better certain phenomena
(see, e.g., \cite{Coimbra,Coimbra2,Diaz,Lorenzo,Od2,Ramirez2,Samko,SamkoRoss}).
To the best of our knowledge, an extension to Hadamard fractional operators is new
and no work has been carried out so far in this direction. This is due to practical difficulties
in computing such fractional derivatives and integrals of variable order. For this reason, here we propose a simple
but effective numerical method that allows to deal with variable fractional order operators of Hadamard type.

The organization of the paper is the following. In Section~\ref{sec:def} we extend
known definitions of Hadamard fractional operators by considering the order $\alpha$ to be
a function, and present a new definition of derivative, the Hadamard--Marchaud fractional derivative,
which is an intrinsic variable order operator.
In Section~\ref{sec:approx} we prove expansion formulas for the given fractional operators,
using only integer-order derivatives. Finally, in Section~\ref{sec:example}
we give some concrete examples of the usefulness of the proposed method,
including the application to the solution of a fractional differential system
of variable order (Section~\ref{sec:application}) and to the solution of
a fractional variational problem of variable order (Section~\ref{sec:appl:fcv}).


\section{Hadamard operators of variable fractional order}
\label{sec:def}

Along the text, the order of fractional operators is given by a function $\alpha\in C^1([a,b],(0,1))$,
and the space of functions $x:[a,b]\to\mathbb{R}$ is such that each of the following integrals
are well-defined, where $a,b$ are two reals with $0<a<b$.

\begin{definition}[Hadamard integrals of variable fractional order]
\label{def:H:i}
The left and right Hadamard fractional integrals of order $\a$ are defined by
$$
\LHI x(t)=\frac{1}{\Gamma(\a)}
\int_a^t\left(\ln\frac{t}{\tau}\right)^{\a-1}\frac{x(\tau)}{\tau}d\tau
$$
and
$$
\RHI x(t)=\frac{1}{\Gamma(\a)}\int_t^b
\left(\ln\frac{\tau}{t}\right)^{\a-1}\frac{x(\tau)}{\tau}d\tau,
$$
respectively.
\end{definition}

\begin{definition}[Hadamard derivatives of variable fractional order]
\label{def:H:d}
The left and right Hadamard fractional derivatives of order $\a$ are defined by
$$
\LHD x(t)=\frac{t}{\Gamma(1-\a)}\frac{d}{dt}
\int_a^t \left(\ln\frac{t}{\tau}\right)^{-\a}\frac{x(\tau)}{\tau}d\tau
$$
and
$$
\RHD x(t)=\frac{-t}{\Gamma(1-\a)}\frac{d}{dt}
\int_t^b \left(\ln\frac{\tau}{t}\right)^{-\a}\frac{x(\tau)}{\tau}d\tau,
$$
respectively.
\end{definition}

The two Definitions~\ref{def:H:i} and \ref{def:H:d}
coincide with the classical definitions of Hadamard
when the order $\a$ is a constant function. Besides these definitions,
we introduce a different one inspired on Hadamard and Marchaud
fractional derivatives \cite{Samko:book}.

\begin{definition}[Hadamard--Marchaud derivatives of variable fractional order]
\label{def:H:M}
The left and right Hadamard--Marchaud fractional derivatives of order $\a$ are defined by
\begin{multline}
\label{eq:left:H:d}
\LHMD x(t)=\frac{x(t)}{\Gamma(1-\a)}\left(\ln\frac{t}{a}\right)^{-\a}\\
+\frac{\a}{\Gamma(1-\a)}\int_a^t\frac{x(t)-x(\tau)}{\tau}\left(\ln\frac{t}{\tau}\right)^{-\a-1}d\tau
\end{multline}
and
\begin{multline}
\label{eq:right:H:d}
\RHMD x(t)=\frac{x(t)}{\Gamma(1-\a)}\left(\ln\frac{b}{t}\right)^{-\a}\\
+\frac{\a}{\Gamma(1-\a)}\int_t^b\frac{x(t)-x(\tau)}{\tau}\left(\ln\frac{\tau}{t}\right)^{-\a-1}d\tau,
\end{multline}
respectively.
\end{definition}

\begin{remark}
Splitting the integrals in Definition~\ref{def:H:M} into two,
and integrating by parts, we get that \eqref{eq:left:H:d} is
equivalent to
\begin{equation}
\label{eq:marchaud}
\begin{split}
\LHMD  x(t)&=- \DS\frac{\a}{\Gamma(1-\a)}\int_a^t x(\t)
\cdot \frac{1}{\t}\left(\ln\frac{t}{\t}\right)^{-\alpha(t)-1}d\t\\
&=\DS\frac{x(a)}{\Gamma(1-\a)}\left(\ln\frac{t}{a}\right)^{-\a}
+\frac{1}{\Gamma(1-\a)}\int_a^t \left(\ln\frac{t}{\t}\right)^{-\alpha(t)}x'(\t)d\t
\end{split}
\end{equation}
while \eqref{eq:right:H:d} is equivalent to
$$
\RHMD x(t) =\frac{x(b)}{\Gamma(1-\a)}\left(\ln\frac{b}{t}\right)^{-\a}
-\frac{1}{\Gamma(1-\a)}\int_t^b \left(\ln\frac{\t}{t}\right)^{-\alpha(t)}x'(\t)d\t.
$$
\end{remark}

Other Hadamard notions of variable fractional order are possible.
For example, motivated by the Caputo fractional derivative, we can set
\begin{equation}
\label{caputoVarorder}
{_a^C\mathcal{D}_t^{\a}}x(t):={_a\mathcal{D}_t^{\a}}(x(t)-x(a)).
\end{equation}
Integrating by parts, we then obtain that
\begin{equation*}
\begin{split}
{{_a^C\mathcal{D}_t^{\a}}}x(t) &=\DS\frac{t}{\Gamma(1-\a)}\frac{d}{dt}
\int_a^t \left(\ln\frac{t}{\tau}\right)^{-\a}\frac{1}{\tau}(x(\tau)-x(a))d\tau\\
&=\DS  \frac{t}{\Gamma(1-\a)}\frac{d}{dt}\left[\frac{1}{1-\a}
\int_a^t \left(\ln\frac{t}{\tau}\right)^{1-\a}x'(\tau)d\tau\right]\\
&=\DS  \frac{t\alpha'(t)}{\Gamma(2-\a)}
\int_a^t \left(\ln\frac{t}{\tau}\right)^{1-\a}x'(\tau)\left[\frac{1}{1-\a}
-\ln\left(\ln\frac{t}{\tau}\right)\right]d\tau\\
&\quad \DS +\frac{1}{\Gamma(1-\a)}
\int_a^t \left(\ln\frac{t}{\tau}\right)^{-\a}x'(\tau)d\tau.
\end{split}
\end{equation*}
Similar calculations can be done for the right Hadamard--Caputo
derivative of variable fractional order,
${_t^C\mathcal{D}_b^{\a}}x(t):={_t\mathcal{D}_b^{\a}}(x(t)-x(a))$.

\begin{definition}[Left Hadamard--Caputo derivative of variable fractional order]
\label{def:H:C}
The left Hadamard--Caputo fractional derivatives of order $\a$  is defined by
\begin{multline*}
{{_a^C\mathcal{D}_t^{\a}}}x(t)
= \DS  \frac{t\alpha'(t)}{\Gamma(2-\a)}
\int_a^t \left(\ln\frac{t}{\tau}\right)^{1-\a}x'(\tau)\left[\frac{1}{1-\a}
-\ln\left(\ln\frac{t}{\tau}\right)\right]d\tau\\
\DS +\frac{1}{\Gamma(1-\a)}
\int_a^t \left(\ln\frac{t}{\tau}\right)^{-\a}x'(\tau)d\tau.
\end{multline*}
\end{definition}

\begin{remark}
If we consider the particular case $\a \equiv \alpha$, $\alpha$ a constant,
then Definition~\ref{def:H:C} simplifies to the Hadamard--Caputo fractional derivative
studied in \cite{Jarad}:
\begin{equation*}
{_a^C\mathcal{D}_t^{\alpha}}x(t)=\frac{1}{\Gamma(1-\alpha)}
\int_a^t \left(\ln\frac{t}{\tau}\right)^{-\alpha}x'(\tau)d\tau.
\end{equation*}
\end{remark}

\begin{lemma}
\label{DIofLn}
Let $\beta>-1$. If
$$
x(t)=\left(\ln\frac{t}{a}\right)^\beta,
$$
then the left Hadamard fractional integral of order $\a$ is given by
$$
\LHI x(t)=\frac{\Gamma(\beta+1)}{\Gamma(\beta
+\a+1)}\left(\ln\frac{t}{a}\right)^{\beta+\a},
$$
the left Hadamard fractional derivative of order $\a$ is given by
\begin{multline*}
\LHD x(t)=\frac{\Gamma(\beta+1)}{\Gamma(\beta
-\a+1)}\left(\ln\frac{t}{a}\right)^{\beta-\a}\\
-\frac{t \alpha'(t) \Gamma(\beta+1)}{\Gamma(\beta
-\a+2)}\left(\ln\frac{t}{a}\right)^{\beta-\a+1}
\left[\ln\left(\ln\frac{t}{a}\right)
+\psi(1-\a)-\psi(\beta-\a+2)\right],
\end{multline*}
where $\psi$ is the Psi function, that is, $\psi$ is
the derivative of the logarithm of the Gamma function,
$$
\psi(t)=\frac{d}{dt} \ln\left({\Gamma(t)}\right)
= \frac{\Gamma'(t)}{\Gamma(t)},
$$
and the left Hadamard--Marchaud fractional
derivative of order $\a$ is given by
$$
\LHMD x(t)=\frac{\Gamma(\beta+1)}{\Gamma(\beta
-\a+1)}\left(\ln\frac{t}{a}\right)^{\beta-\a}.
$$
\end{lemma}

\begin{proof}
Starting with Definition~\ref{def:H:i}, we arrive to
$$
\LHI x(t)=\frac{1}{\Gamma(\a)}\left(\ln\frac{t}{a}\right)^{\a-1}
\int_a^t \left(1-\frac{\ln\frac{\t}{a}}{\ln\frac{t}{a}}\right)^{\a-1}\left(
\ln\frac{\t}{a}\right)^{\beta}\frac{d\t}{\t}.
$$
Performing the change of variable
$$
\ln\frac{\t}{a}=s\ln\frac{t}{a},
$$
it follows that
\begin{equation*}
\begin{split}
\LHI x(t)&=\DS\frac{1}{\Gamma(\a)}\left(
\ln\frac{t}{a}\right)^{\beta+\a}\int_0^1(1-s)^{\a-1}s^{\beta}ds\\
&=\DS\frac{1}{\Gamma(\a)}\left(\ln\frac{t}{a}\right)^{\beta+\a}B(\a,\beta+1)\\
&=\DS\frac{1}{\Gamma(\a)}\left(\ln\frac{t}{a}\right)^{\beta+\a}
\frac{\Gamma(\a)\Gamma(\beta+1)}{\Gamma(\beta+\a+1)}\\
&=\DS\frac{\Gamma(\beta+1)}{\Gamma(\beta+\a+1)}\left(
\ln\frac{t}{a}\right)^{\beta+\a},
\end{split}
\end{equation*}
where $B$ is the Beta function, that is,
$$
B(\lambda,\mu)=\int_0^1t^{\lambda-1}(1-t)^{\mu-1}\, dt,
\quad \lambda, \mu > 0.
$$
The formula for $\LHMD x(t)$ is obtained in a similar way,
using \eqref{eq:marchaud}:
\begin{equation*}
\begin{split}
\LHMD x(t)&=\DS\frac{\beta}{\Gamma(1-\a)}\left(\ln\frac{t}{a}\right)^{\beta-\a}B(1-\a,\beta)\\
&=\DS \frac{\Gamma(\beta+1)}{\Gamma(\beta-\a+1)}\left(\ln\frac{t}{a}\right)^{\beta-\a}.
\end{split}
\end{equation*}
To prove the formula for the left Hadamard fractional derivative of order $\a$,
we start with the same change of variables as before, to get
\begin{equation}
\label{aux1}
\LHD x(t)=\frac{t}{\Gamma(1-\a)}\frac{d}{dt}\left[
\left(\ln\frac{t}{a}\right)^{\beta-\a+1}\frac{\Gamma(1-\a)
\Gamma(\beta+1)}{\Gamma(\beta-\a+2)}\right].
\end{equation}
The intended formula follows directly by computing the derivative in \eqref{aux1}.
\end{proof}

As a consequence of Lemma~\ref{DIofLn}, we have that
$\LHD x(t)\not= \LHMD x(t)$. Next, we establish a
relation between these two types of differential operators.

\begin{theorem}
\label{thm:rel}
The following relation between the left Hadamard
and the left Hadamard--Marchaud fractional
derivatives of order $\a$ holds:
$$
\LHD x(t)=\LHMD x(t)-\frac{t \alpha'(t)}{\Gamma(1-\a)}
\int_a^t \ln\left(\ln \frac{t}{\t}\right)\left(
\ln \frac{t}{\t}\right)^{-\a}\frac{x(\t)}{\t}d\t.
$$
\end{theorem}

\begin{proof}
Starting with the definition,
and differentiating the integral, we obtain that
\begin{equation*}
\begin{split}
\LHD x(t)&=\DS\frac{t}{\Gamma(1-\a)}\frac{d}{dt}
\int_a^t \left(\ln\frac{t}{\tau}\right)^{-\a}\frac{x(\tau)}{\tau}d\tau\\
&\DS=\frac{t}{\Gamma(1-\a)}\lim_{\epsilon\to0}\frac{d}{dt}
\int_a^{te^{-\epsilon}} \left(\ln\frac{t}{\tau}\right)^{-\a}\frac{x(\tau)}{\tau}d\tau\\
&\DS=\frac{t}{\Gamma(1-\a)}\lim_{\epsilon\to0}\left[\frac{x(te^{-\epsilon})}{t\epsilon^{\a}}\right.\\
&\left.\DS\quad +\int_a^{te^{-\epsilon}} \left(\ln\frac{t}{\tau}\right)^{-\a}\frac{x(\tau)}{\tau}\left[
-\alpha'(t)\ln\left(\ln \frac{t}{\t}\right)-\frac{\a}{t \ln\frac{t}{\t}}\right]d\t\right]\\
&\DS=\frac{1}{\Gamma(1-\a)}\lim_{\epsilon\to0}\left[\frac{x(te^{-\epsilon})}{\epsilon^{\a}}-\a
\int_a^{te^{-\epsilon}} \left(\ln\frac{t}{\tau}\right)^{-\a-1}\frac{x(\tau)-x(t)}{\tau}d\tau\right.\\
&\quad \left. \DS  -\a\int_a^{te^{-\epsilon}} \left(\ln\frac{t}{\tau}\right)^{-\a-1}\frac{x(t)}{\tau}d\tau\right] \\
&\quad -\DS\alpha'(t)\frac{t}{\Gamma(1-\a)}\int_a^t
\ln\left(\ln \frac{t}{\t}\right)\left(\ln \frac{t}{\t}\right)^{-\a}\frac{x(\t)}{\t} d\t.
\end{split}
\end{equation*}
Integrating,
$$
\int_a^{te^{-\epsilon}} \left(\ln\frac{t}{\tau}\right)^{-\a-1}\frac{x(t)}{\tau}d\tau
=\frac{x(t)}{\a}\left(\epsilon^{-\a}-\left(\ln\frac{t}{a}\right)^{-\a}\right).
$$
Also, since $\a\in(0,1)$, then
$$
\lim_{\epsilon\to0}\frac{x(te^{-\epsilon})-x(t)}{\epsilon^{\a}}=0.
$$
Using these two relations, we prove that
\begin{equation*}
\begin{split}
\LHD x(t)&=\DS\frac{1}{\Gamma(1-\a)}\lim_{\epsilon\to0}\frac{x(te^{-\epsilon})-x(t)}{\epsilon^{\a}}\\
& \quad \DS +\frac{x(t)}{\Gamma(1-\a)}\left(\ln\frac{t}{a}\right)^{-\a}\\
& \quad +\frac{\a}{\Gamma(1-\a)}\int_a^{t} \left(\ln\frac{t}{\tau}\right)^{-\a-1}\frac{x(t)-x(\t)}{\tau}d\tau\\
&\DS\quad-\alpha'(t)\frac{t}{\Gamma(1-\a)}\int_a^t \ln\left(
\ln \frac{t}{\t}\right)\left(\ln \frac{t}{\t}\right)^{-\a}\frac{x(\t)}{\t} d\t.\\
&\DS =\LHMD x(t)-\alpha'(t)\frac{t}{\Gamma(1-\a)}\int_a^t \ln\left(
\ln \frac{t}{\t}\right)\left(\ln \frac{t}{\t}\right)^{-\a}\frac{x(\t)}{\t}d\t.
\end{split}
\end{equation*}
\end{proof}

The next corollary is a trivial consequence of Theorem~\ref{thm:rel}.
It asserts that it only makes sense to distinguish between
Hadamard and Hadamard--Marchaud fractional derivatives
in the variable-order case: for the classical situation
of constant order $\alpha$, the fractional derivatives coincide.

\begin{corollary}
If $\a \equiv \, const$, then both left Hadamard
and left Hadamard--Marchaud fractional derivatives coincide.
\end{corollary}


\section{Approximations}
\label{sec:approx}

In this section we exhibit several results about approximations for
the Hadamard fractional operators, which are expressed by only using
integer-order derivatives of the function. With this in hand, given any
problem that depends on these fractional operators, we are able
to rewrite it by eliminating all the fractional operators, and by doing
so obtaining a classical problem, with dependence on integer-order derivatives only.
Then one can apply any known technique from the literature.
We mention \cite{Atanackovic1}, where an analogous idea was carried out for
the Riemann--Liouville fractional derivative. In the following, given
$k\in\mathbb N \cup\{0\}$, we define the sequences $x_{k,0}(t)$ and
$x_{k,1}(t)$ recursively by the formulas
$$
x_{0,0}(t)=x(t),
\quad x_{k+1,0}(t) =t\frac{d}{dt} x_{k,0}(t),
\mbox{ for } k\in\mathbb N\cup\{0\},
$$
and
$$
x_{0,1}(t)=x'(t),
\quad x_{k+1,1}(t) =\frac{d}{dt} (t x_{k,1}(t)),
\mbox{ for } k\in\mathbb N\cup\{0\}.
$$

The following definition is useful to describe our approximations.

\begin{definition}[Left and right moment of a function]
The left moment of $x$ of order $k\in\mathbb N$ is given by
$$
V_k(t) =(k-n)\int_a^t \left(\ln\frac{\t}{a}\right)^{k-n-1}\frac{x(\t)}{\t}d\t,
$$
and the right moment of $x$ of order $k\in\mathbb N$ by
$$
W_k(t) =(k-n)\int_t^b \left(\ln\frac{b}{\t}\right)^{k-n-1}\frac{x(\t)}{\t}d\t.
$$
\end{definition}

\begin{theorem}
\label{MainthmIntegral}
Fix $n \in \mathbb{N}$ and $N\geq n+1$, and let
$x(\cdot)\in C^{n+1}([a,b],\mathbb{R})$, where  $0<a<b$. Then,
\begin{multline}
\label{eq:approx}
\LHI x(t)=\sum_{k=0}^{n}A(k)\left(\ln\frac{t}{a}\right)^{\a+k} x_{k,0}(t)\\
+\sum_{k=n+1}^N B(k)\left(\ln\frac{t}{a}\right)^{\a+n-k}V_k(t)+E_N(t)
\end{multline}
with
\begin{equation*}
\begin{split}
A(k)&=\frac{1}{\Gamma(\a+k+1)}\left[1+\sum_{p=n-k+1}^N
\frac{\Gamma(p-\a-n)}{\Gamma(-\a-k)(p-n+k)!}\right],\quad k = 0, \ldots, n,\\
B(k)&=\frac{\Gamma(k-\a-n)}{\Gamma(\a)\Gamma(1-\a)(k-n)!}, \quad k=n+1,\ldots,N.
\end{split}
\end{equation*}
Relation \eqref{eq:approx} gives an approximation for the
left Hadamard fractional integral of order $\a$ with
error $E_{N}(t)$ bounded by
$$
|E_{N}(t)|\leq \max_{\tau\in[a,t]}|x_{n,1}(\tau)|
\frac{\exp((n+\a)^2+n+\a)}{\Gamma(n+1+\a)(n+\a)N^{n+\a}}\left(\ln\frac{t}{a}\right)^{n+\a}(t-a).
$$
\end{theorem}

\begin{proof}
Similar to the one given in \cite{PAT},
replacing $\alpha$ by $\alpha(t)$.
\end{proof}

\begin{theorem}
\label{MainthmHderivative}
Fix $n \in \mathbb{N}$ and $N\geq n+1$, and let $x(\cdot)\in C^{n+1}([a,b],\mathbb{R})$,
where  $0<a<b$. Then,
$$
\LHD x(t)=S_1(t)-S_2(t)+E_{1,N}(t)+E_{2,N}(t)
$$
with
\begin{equation}
\label{eq:S1}
S_1(t)=\sum_{k=0}^{n}A(k)\left(\ln\frac{t}{a}\right)^{k-\a} x_{k,0}(t)
+\sum_{k=n+1}^N B(k)\left(\ln\frac{t}{a}\right)^{n-\a-k}V_k(t),
\end{equation}
where
\begin{equation*}
\begin{split}
A(k)&=\frac{1}{\Gamma(k+1-\a)}\left[1+\sum_{p=n-k+1}^N
\frac{\Gamma(p+\a-n)}{\Gamma(\a-k)(p-n+k)!}\right],\quad k = 0, \ldots, n,\\
B(k)&=\frac{\Gamma(k+\a-n)}{\Gamma(-\a)\Gamma(1+\a)(k-n)!}, \quad k=n+1,\ldots,N,
\end{split}
\end{equation*}
and
\begin{equation}
\label{eq:S2}
\begin{split}
S_2(t)&=\frac{tx(t)\alpha'(t)}{\Gamma(1-\a)}\left(
\ln\frac{t}{a}\right)^{1-\a}\left[\frac{\ln\left(
\ln\frac{t}{a}\right)}{1-\a}-\frac{1}{(1-\a)^2}\right.\\
& \quad \left.-\ln\left(\ln\frac{t}{a}\right)\sum_{k=0}^N \C\frac{(-1)^k}{k+1}
+\sum_{k=0}^N\C(-1)^k\sum_{p=1}^N\frac{1}{p(k+p+1)}\right]\\
&\quad +\frac{t\alpha'(t)}{\Gamma(1-\a)}\left(\ln\frac{t}{a}\right)^{1-\a}\\
&\quad \times \left[\ln\left(\ln\frac{t}{a}\right)\sum_{k=n+1}^{N+n+1}
\D\frac{(-1)^{k-n-1}}{k-n}\left(\ln\frac{t}{a}\right)^{n-k}V_{k}(t)\right.\\
&\quad \left.-\sum_{k=n+1}^{N+n+1}\D(-1)^{k-n-1}\sum_{p=1}^N\frac{1}{p(k+p-n)}\left(
\ln\frac{t}{a}\right)^{n-k-p} V_{k+p}(t)\right].
\end{split}
\end{equation}
The error $E_{1,N}(t)+E_{2,N}(t)$ of approximating
the left Hadamard fractional derivative
$\LHD x(t)$ by $S_1(t)-S_2(t)$ is bounded with
\begin{multline}
\label{eq:err1}
|E_{1,N}(t)|\\
\leq \max_{\tau\in[a,t]}|x_{n,1}(\tau)|
\frac{\exp((n-\a)^2+n-\a)}{\Gamma(n+1-\a)(n-\a)N^{n-\a}}
\left(\ln\frac{t}{a}\right)^{n-\a}(t-a)
\end{multline}
and
\begin{multline}
\label{eq:err2}
|E_{2,N}(t)|\\
\leq  \max_{\tau\in[a,t]}|x'(\tau)|
\frac{\left|t(2t-a)\alpha'(t)\left(\ln\frac{t}{a}\right)^{2-\alpha(t)}\right|
\exp(\alpha^2(t)-\a)}{\Gamma(2-{\a}){N^{1-\a}}}
\left[\left|\ln\left(\ln\frac{t}{a}\right)\right|+\frac{1}{N}\right].
\end{multline}
\end{theorem}

\begin{proof}
Doing the change of variables
$$
\frac{t}{\t}=\frac{u}{a},
$$
and some calculations needed, we deduce that
\begin{equation*}
\begin{split}
_a&\mathcal{D}_t^{\a} x(t) =\DS\displaystyle\frac{t}{\Gamma(1-\a)}\frac{d}{dt}
\int_a^t \left(\ln\frac{u}{a}\right)^{-\alpha(t)}x\left(\frac{at}{u}\right)\frac{1}{u} du\\
&=\DS\frac{t}{\Gamma(1-\a)}\left[\frac{x(a)}{t}\left(\ln\frac{t}{a}\right)^{-\a}\right.\\
&\quad \left.+\DS\int_a^t  \left[ -\alpha'(t)\left(\ln\frac{u}{a}\right)^{-\alpha(t)}
\ln\left(\ln\frac{u}{a}\right)x\left(\frac{at}{u}\right)\frac{1}{u}
+\left(\ln\frac{u}{a}\right)^{-\a}x'\left(\frac{at}{u}\right)\frac{a}{u^2}\right]du\right]\\
&=\DS\frac{t}{\Gamma(1-\a)}\left[\frac{x(a)}{t}\left(\ln\frac{t}{a}\right)^{-\a}\right.\\
&\quad \left.\DS+\int_a^t  \left[ -\alpha'(t)\left(\ln\frac{t}{\t}\right)^{-\alpha(t)}
\ln\left(\ln\frac{t}{\t}\right)\frac{x(\t)}{\t}
+\left(\ln\frac{t}{\t}\right)^{-\a}\frac{x'(\t)}{t}\right]d\t\right].
\end{split}
\end{equation*}
Define
$$
\overline S_1(t)=\frac{x(a)}{\Gamma(1-\a)}\left(\ln\frac{t}{a}\right)^{-\a}
+\frac{1}{\Gamma(1-\a)}\int_a^t \left(\ln\frac{t}{\t}\right)^{-\alpha(t)}x'(\t)d\t
$$
and
$$
\overline S_2(t)=\frac{t\alpha'(t)}{\Gamma(1-{\a})}
\int_a^t \left(\ln\frac{t}{\t}\right)^{-\alpha(t)}\ln\left(\ln\frac{t}{\t}\right)\frac{x(\t)}{\t}d\t.
$$
When $\a=\alpha$, i.e., when $\a$ is constant, formula $\overline S_1(t)$ is equivalent
to the left Hadamard fractional derivative \eqref{classicalHD} (see \cite{Kilbas2}),
and following the same techniques as the ones given in \cite{PAT}, we obtain formula
$S_1(t)$ as in \eqref{eq:S1} and the upper bound formula for $|E_{1,N}(t)|$ as in \eqref{eq:err1}.
About the sum $\overline S_2(t)$, starting with the relation
$$
\overline S_2(t)=\frac{t\alpha'(t)}{\Gamma(1-{\a})}\int_a^t x(\t)
\cdot\left[\frac{d}{d\t} \int_a^\t \left(\ln\frac{t}{u}\right)^{-\alpha(t)}
\ln\left(\ln\frac{t}{u}\right)\frac{du}{u}\right]d\t,
$$
and performing integration by parts, we get
\begin{multline*}
\overline S_2(t)=\frac{t\alpha'(t)}{\Gamma(1-{\a})}\left[x(t)
\int_a^t \left(\ln\frac{t}{u}\right)^{-\alpha(t)}
\ln\left(\ln\frac{t}{u}\right)\frac{du}{u}\right.\\
-\left.\int_a^t x'(\t)\left[\int_a^\t \left(\ln\frac{t}{u}\right)^{-\alpha(t)}
\ln\left(\ln\frac{t}{u}\right)\frac{du}{u}\right]d\t\right].
\end{multline*}
From simple computations, we have
$$
\int_a^t \left(\ln\frac{t}{u}\right)^{-\alpha(t)}
\ln\left(\ln\frac{t}{u}\right)\frac{du}{u}
=\left(\ln\frac{t}{a}\right)^{1-\a}\left[
\frac{\ln\left(\ln\frac{t}{a}\right)}{1-\a}-\frac{1}{(1-\a)^2}\right].
$$
Also, by Taylor's theorem, we have the two following formulas:
\begin{equation*}
\begin{split}
\left(\ln\frac{t}{u}\right)^{-\alpha(t)}
&= \left(\ln\frac{t}{a}\right)^{-\alpha(t)}\left(1
-\frac{\ln\frac{u}{a}}{\ln\frac{t}{a}}\right)^{-\alpha(t)}\\
&=\left(\ln\frac{t}{a}\right)^{-\alpha(t)}\sum_{k=0}^N \C (-1)^k
\frac{\left(\ln\frac{u}{a}\right)^k}{\left(\ln\frac{t}{a}\right)^k}+E'_N(t)
\end{split}
\end{equation*}
and
$$
\ln\left(\ln\frac{t}{u}\right)
=\ln\left(\ln\frac{t}{a}\right)
+\ln\left(1-\frac{\ln\frac{u}{a}}{\ln\frac{t}{a}}\right)
=\ln\left(\ln\frac{t}{a}\right)-\sum_{p=1}^N \frac{1}{p}
\frac{\left(\ln\frac{u}{a}\right)^p}{\left(\ln\frac{t}{a}\right)^p}+E''_N(t).
$$
Combining all the previous equalities, we get
\begin{equation*}
\begin{split}
&\overline{S}_2(t)=\DS\frac{t\alpha'(t)}{\Gamma(1-{\a})}\left[
x(t)\left(\ln\frac{t}{a}\right)^{1-\alpha(t)}
\left[\frac{\ln\left(\ln\frac{t}{a}\right)}{1-\a}-\frac{1}{(1-\a)^2}\right]\right.\\
&\quad \DS-\int_a^t x'(\t)\left(\ln\frac{t}{a}\right)^{-\alpha(t)}\ln\left(\ln\frac{t}{a}\right)
\sum_{k=0}^N \C \frac{(-1)^k}{\left(\ln\frac{t}{a}\right)^k}
\left(\int_a^\t \left(\ln\frac{u}{a}\right)^k \, \frac{du}{u}\right)d\t\\
&\quad \DS\left. +\int_a^t x'(\t) \left(\ln\frac{t}{a}\right)^{-\alpha(t)}
\sum_{k=0}^N \C\frac{(-1)^k}{\left(\ln\frac{t}{a}\right)^k} \,
\sum_{p=1}^N \frac{\int_a^\t \left(\ln\frac{u}{a}\right)^{k+p}
\, \frac{du}{u}}{p\left(\ln\frac{t}{a}\right)^p}\,d\t\right]+E_{2,N}(t)\\
&=\quad \DS\frac{t\alpha'(t)}{\Gamma(1-{\a})}\left(
\ln\frac{t}{a}\right)^{-\a}\left[x(t)\ln\frac{t}{a}
\left[\frac{\ln\left(\ln\frac{t}{a}\right)}{1-\a}-\frac{1}{(1-\a)^2}\right]\right.\\
&\quad \DS-\ln\left(\ln\frac{t}{a}\right)\sum_{k=0}^N \C \frac{(-1)^k}{\left(
\ln\frac{t}{a}\right)^k(k+1)}\int_a^t x'(\t)\left(\ln\frac{\t}{a}\right)^{k+1}d\t\\
&\quad \DS\left. +\sum_{k=0}^N \C\frac{(-1)^k}{\left(\ln\frac{t}{a}\right)^k} \,
\sum_{p=1}^N \frac{\int_a^t x'(\t)
\left(\ln\frac{\t}{a}\right)^{k+p+1}\,d\t}{p\left(\ln\frac{t}{a}\right)^p(k+p+1)}\right]+E_{2,N}(t).
\end{split}
\end{equation*}
Formula \eqref{eq:S2} is deduced using relations
$$
\int_a^t x'(\t)\left(\ln\frac{\t}{a}\right)^{k+1}\, d\t
= x(t)\left(\ln\frac{t}{a}\right)^{k+1}-V_{k+n+1}(t)
$$
and
$$
\int_a^t x'(\t)\left(\ln\frac{\t}{a}\right)^{k+p+1}\, d\t
= x(t)\left(\ln\frac{t}{a}\right)^{k+p+1}-V_{k+p+n+1}(t).
$$
The error that occurs on this approximation is bounded by
\begin{equation*}
\begin{split}
|E_{2,N}(t)|&\leq \left| \frac{t\alpha'(t)}{
\Gamma(1-{\a})}\left(\ln\frac{t}{a}\right)^{-\a} \right|\\
&\quad \times\left|-\ln\left(\ln\frac{t}{a}\right)\sum_{k=N+1}^\infty
\C \frac{(-1)^k}{k+1}\int_a^t x'(\t)\frac{\left(
\ln\frac{\t}{a}\right)^{k+1}}{\left(\ln\frac{t}{a}\right)^k}d\t\right.\\
& \quad \left. +\sum_{k=N+1}^\infty \C(-1)^k \,\sum_{p=N+1}^\infty
\frac{1}{p(k+p+1)}\int_a^t x'(\t) \frac{\left(
\ln\frac{\t}{a}\right)^{k+p+1}}{\left(\ln\frac{t}{a}\right)^{k+p}}\,d\t\right|.
\end{split}
\end{equation*}
Since
$$
\left|\C\right|\leq\frac{\exp(\alpha^2(t)-\a)}{k^{1-\a}},
$$
it follows that
$$
\begin{array}{ll}
|E_{2,N}(t)|&\leq \DS\max_{\tau\in[a,t]}|x'(\tau)|
\left| \frac{t\alpha'(t)}{\Gamma(1-{\a})}\left(
\ln\frac{t}{a}\right)^{-\a} \right|\exp(\alpha^2(t)-\a)\\
&\DS\times\left[\left|\ln\left(\ln\frac{t}{a}\right)\right|
\sum_{k=N+1}^\infty \frac{1}{k^{1-\a}(k+1)}\int_a^t
\frac{\left(\ln\frac{\t}{a}\right)^{k+1}}{\left(\ln\frac{t}{a}\right)^k}d\t\right.\\
&\DS\left. +\sum_{k=N+1}^\infty \,\sum_{p=N+1}^\infty \frac{1}{k^{1-\a}p(k+p+1)}
\int_a^t  \frac{\left(\ln\frac{\t}{a}\right)^{k+p+1}}{\left(\ln\frac{t}{a}\right)^{k+p}}\,d\t\right].
\end{array}
$$
Integrating by parts,
\begin{equation*}
\begin{split}
\int_a^t\frac{\left(\ln\frac{\t}{a}\right)^{k+1}}{\left(\ln\frac{t}{a}\right)^k}d\t
&=\frac{1}{\left(\ln\frac{t}{a}\right)^k}\int_a^t\tau
\cdot \left(\ln\frac{\t}{a}\right)^{k+1}\frac{d\t}{\t}\\
&=\frac{t\left(\ln\frac{t}{a}\right)^2}{k+2}-\frac{1}{k+2}
\int_a^t\frac{\left(\ln\frac{\t}{a}\right)^{k+2}}{\left(
\ln\frac{t}{a}\right)^k}d\t\\
&\leq \frac{(2t-a)\left(\ln\frac{t}{a}\right)^2}{k+2}
\end{split}
\end{equation*}
by the inequality
$$
0 \leq \frac{\left(\ln\frac{\t}{a}\right)^{k+2}}{\left(\ln\frac{t}{a}\right)^k}
\leq \left(\ln\frac{t}{a}\right)^2.
$$
Similarly, it can be proven that
$$
\int_a^t\frac{\left(\ln\frac{\t}{a}\right)^{k+p+1}}{\left(\ln\frac{t}{a}\right)^{k+p}}d\t
\leq \frac{(2t-a)\left(\ln\frac{t}{a}\right)^2}{k+p+2}.
$$
Then,
$$
\begin{array}{ll}
|E_{2,N}(t)|&\leq \DS\max_{\tau\in[a,t]}|x'(\tau)|
\left| \frac{t(2t-a)\alpha'(t)}{\Gamma(1-{\a})}\left(
\ln\frac{t}{a}\right)^{2-\a} \right|\exp(\alpha^2(t)-\a)\\
&\DS\times\left[\left|\ln\left(\ln\frac{t}{a}\right)\right|
\int_{N}^\infty \frac{1}{k^{1-\a}(k+1)(k+2)}dk\right.\\
&\DS\left. +\int_{N}^\infty \,\int_{N}^\infty
\frac{1}{k^{1-\a}p(k+p+1)(k+p+2)}dp\,dk\right].
\end{array}
$$
Finally, since
$$
\int_N^\infty\frac{1}{k^{1-\a}(k+1)(k+2)}\,dk
<\int_N^\infty\frac{1}{k^{2-\a}}\,dk=\frac{1}{(1-\a)N^{1-\a}}
$$
and
\begin{equation*}
\begin{split}
\int_N^\infty\int_N^\infty\frac{1}{k^{1-\a}p(k+p+1)(k+p+2)}\,dp\, dk
&< \int_N^\infty\int_N^\infty\frac{1}{k^{2-\a}p^2}\,dp\, dk\\
&=\frac{1}{(1-\a)N^{2-\a}},
\end{split}
\end{equation*}
formula \eqref{eq:err2} follows.
\end{proof}

\begin{remark}
Taking into consideration \eqref{caputoVarorder} and Lemma~\ref{DIofLn},
a similar approximation to that given by Theorem~\ref{MainthmHderivative}
can be deduced for the Hadamard--Caputo fractional derivative.
\end{remark}

At last, we consider now the left Hadamard--Marchaud fractional derivative.
As we will see, the expansion formula is similar to the Hadamard
fractional integral, just replacing $\a$ by $-\a$.

\begin{theorem}
\label{MainthmHMderivative}
Fix $n \in \mathbb{N}$ and $N\geq n+1$, and let
$x(\cdot)\in C^{n+1}([a,b],\mathbb{R})$, where  $0<a<b$.
Then,
$$
\LHMD x(t)=S_1(t)+E_{1,N}(t)
$$
with $S_1(t)$ and $E_{1,N}(t)$ given as
in Theorem~\ref{MainthmHderivative}.
\end{theorem}

\begin{proof}
It follows from relation \eqref{eq:marchaud}
and Theorem~\ref{MainthmHderivative}.
\end{proof}

With simple modifications, similar formulas for
the right fractional operators can be deduced.
In case of the right Hadamard fractional integral
$\RHI x(t)$ and of the right Hadamard--Marchaud fractional derivative
$\RHMD x(t)$, they are similar to the ones proved in \cite{PAT}, replacing
$\alpha$ by $\a$. For the right Hadamard fractional derivative $\RHD x(t)$,
the expansion formula is
$$
\RHD x(t)=S_1(t)+S_2(t)+E_{1,N}(t)+E_{2,N}(t)
$$
with
$$
S_1(t)=\sum_{k=0}^{n}A(k)\left(\ln\frac{b}{t}\right)^{k-\a} x_{k,0}(t)
+\sum_{k=n+1}^N B(k)\left(\ln\frac{b}{t}\right)^{n-\a-k}W_k(t),
$$
where
\begin{equation*}
\begin{split}
A(k)&=\frac{(-1)^k}{\Gamma(k+1-\a)}\left[1+\sum_{p=n-k+1}^N
\frac{\Gamma(p+\a-n)}{\Gamma(\a-k)(p-n+k)!}\right],\quad k = 0, \ldots, n,\\
B(k)&=\frac{\Gamma(k+\a-n)}{\Gamma(-\a)\Gamma(1+\a)(k-n)!}, \quad k=n+1,\ldots,N,
\end{split}
\end{equation*}
and
\begin{equation*}
\begin{split}
S_2(t)&=\frac{tx(t)\alpha'(t)}{\Gamma(1-\a)}\left(\ln\frac{b}{t}\right)^{1-\a}\left[
\frac{\ln\left(\ln\frac{b}{t}\right)}{1-\a}-\frac{1}{(1-\a)^2}\right.\\
&\left.-\ln\left(\ln\frac{b}{t}\right)\sum_{k=0}^N \C\frac{(-1)^k}{k+1}
+\sum_{k=0}^N\C(-1)^k\sum_{p=1}^N\frac{1}{p(k+p+1)}\right]\\
&+\frac{t\alpha'(t)\left(\ln\frac{b}{t}\right)^{1-\a}}{\Gamma(1-\a)}\left[
\ln\left(\ln\frac{b}{t}\right)\sum_{k=n+1}^{N+n+1}
\D\frac{(-1)^{k-n-1}}{k-n}\left(\ln\frac{b}{t}\right)^{n-k}W_{k}(t)\right.\\
&\left.-\sum_{k=n+1}^{N+n+1}\D(-1)^{k-n-1}\sum_{p=1}^N\frac{1}{p(k+p-n)}\left(
\ln\frac{b}{t}\right)^{n-k-p} W_{k+p}(t)\right].
\end{split}
\end{equation*}
Also,
$$
|E_{1,N}(t)|\leq \max_{\tau\in[t,b]}|x_{n,1}(\tau)| \frac{\exp((n-\a)^2+n-\a)}{
\Gamma(n+1-\a)(n-\a)N^{n-\a}}\left(\ln\frac{b}{t}\right)^{n-\a}(b-t)
$$
and
$$
|E_{2,N}(t)|\leq  \max_{\tau\in[t,b]}|x'(\tau)|
\frac{\left|bt\alpha'(t)\left(\ln\frac{b}{t}\right)^{2-\alpha(t)}\right|
\exp(\alpha^2(t)-\a)}{\Gamma(2-{\a}){N^{1-\a}}}
\left[\left|\ln\left(\ln\frac{b}{t}\right)\right|+\frac{1}{N}\right].
$$


\section{Examples and applications}
\label{sec:example}

We begin by testing the efficiency of the given procedure with a concrete example.
For simplicity, we consider the same order $\DS\a=\frac{t}{10}$ and
the same test function $x(t)=(\ln t)^2$, for $t\in[1,5]$, for all examples below.
Using Lemma~\ref{DIofLn}, we know the exact expressions of the fractional operators of $x$, to know:
$$
{_1\mathcal{I}_t^{\a}} x(t)
=\frac{2}{\Gamma\left(3+\frac{t}{10}\right)}(\ln t)^{2+\frac{t}{10}},
$$
\begin{multline*}
{_1\mathcal{D}_t^{\a}} x(t)
=\frac{2}{\Gamma\left(3-\frac{t}{10}\right)}(\ln t)^{2-\frac{t}{10}}\\
-\frac{t}{5\Gamma\left(4-\frac{t}{10}\right)}(\ln t)^{3-\frac{t}{10}}
\left[\ln(\ln t)+\psi\left(1-\frac{t}{10}\right)
-\psi\left(4-\frac{t}{10}\right)\right],
\end{multline*}
and
$$
{_1\mathbb{D}_t^{\a}} x(t)=\frac{2}{\Gamma\left(3
-\frac{t}{10}\right)}(\ln t)^{2-\frac{t}{10}}.
$$
For example, from Theorem~\ref{MainthmIntegral}
the approximation formulas for
${_1\mathcal{I}_t^{\a}} x(t)$, when considering
the cases $n=1$ and $N=2,3,4$, are the following:
$$
{_1\mathcal{I}_t^{\a}} x(t) \approx \frac {( \ln t) ^{2+\frac{t}{10}} \left({t}^{2}-20\,t
+300 \right) }{300 \Gamma\left(2+\frac{t}{10}\right)} \, ,
\quad N=2,
$$
$$
{_1\mathcal{I}_t^{\a}} x(t) \approx
{\frac { (\ln t) ^{2+\frac{t}{10}} \left( -t^3 +40 t^2 -700\,t
+12000 \right) }{12000 \Gamma\left( 2+\frac{t}{10} \right) }},
\quad N=3,
$$
$$
{_1\mathcal{I}_t^{\a}} x(t) \approx
{\frac{(\ln t) ^{2+\frac{t}{10}}\left(t^4-70t^3+1900t^2-33000t
+600000\right)}{600000 \Gamma\left(2+\frac{t}{10}\right)}},
\quad N=4.
$$
In Figures~\ref{BaN1}--\ref{BaN3} we show the plots of the exact expressions of
${_1\mathcal{I}_t^{\a}} x(t)$, ${_1\mathcal{D}_t^{\a}} x(t)$ and ${_1\mathbb{D}_t^{\a}} x(t)$
and some approximations of them, for different values of $N$ and $n$. We first fix $n=1$
and consider $N$ variable, with the values $N=2,3,4$, and also the case $N=4$ fixed,
and $n=1,2,3$ variable. The error $E$ is measured using the $L^2$ norm:
\begin{equation}
\label{eq:L2norm}
E(f,g)=\left(\int_1^5 (f(t)-g(t))^2\,dt\right)^{\frac12}.
\end{equation}
\begin{figure}[ht]
\centering
\psfrag{x\(t\)=ln\(t\)^2}{\tiny \hspace*{-0.3cm} ${_1\mathcal{I}_t^{\a}} x(t)$}
\psfrag{t}{\tiny $t$}
\subfigure[$n=1$]{\includegraphics[scale=0.35,angle =-90]{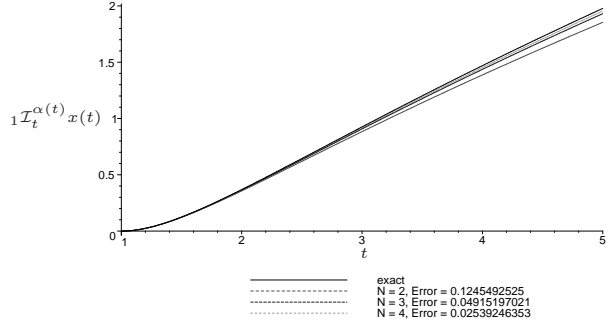}}
\psfrag{x\(t\)=ln\(t\)^2}{ \tiny  \hspace*{-0.3cm} ${_1\mathcal{I}_t^{\a}} x(t)$}
\psfrag{t}{\tiny $t$}
\subfigure[$N=4$]{\includegraphics[scale=0.35,angle =-90]{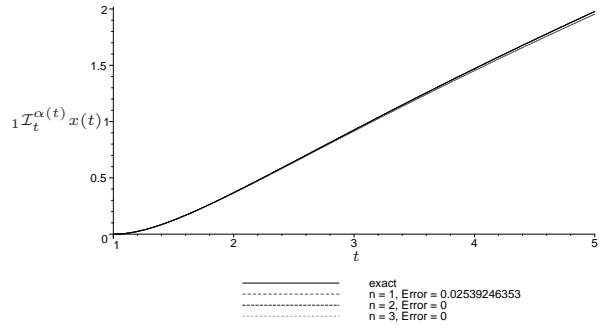}}
\caption{Exact and numerical approximations obtained from Theorem~\ref{MainthmIntegral}
for the left Hadamard integral of variable fractional order
${_1\mathcal{I}_t^{\a}} x(t)$ with $\a=\frac{t}{10}$
and $x(t)=(\ln t)^2$, $t\in[1,5]$.}
\label{BaN1}
\end{figure}
\begin{figure}[ht]
\begin{center}
\psfrag{x\(t\)=ln\(t\)^2}{\tiny  \hspace*{-0.35cm} ${_1\mathcal{D}_t^{\a}} x(t)$}
\psfrag{t}{\tiny $t$}
\subfigure[$n=1$]{\includegraphics[scale=0.35,angle =-90]{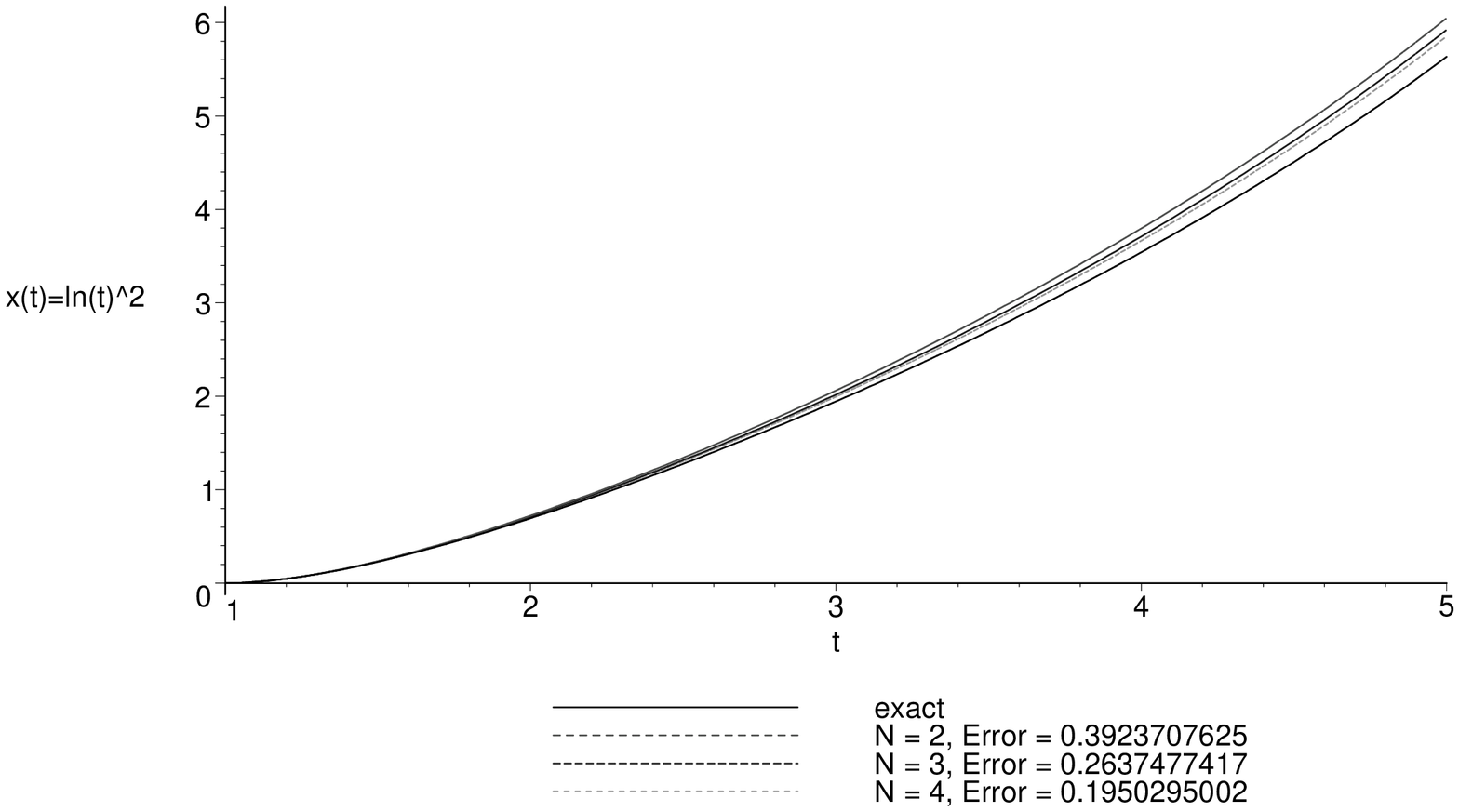}}
\psfrag{x\(t\)=ln\(t\)^2}{\tiny  \hspace*{-0.35cm} ${_1\mathcal{D}_t^{\a}} x(t)$}
\psfrag{t}{\tiny $t$}
\subfigure[$N=4$]{\includegraphics[scale=0.35,angle =-90]{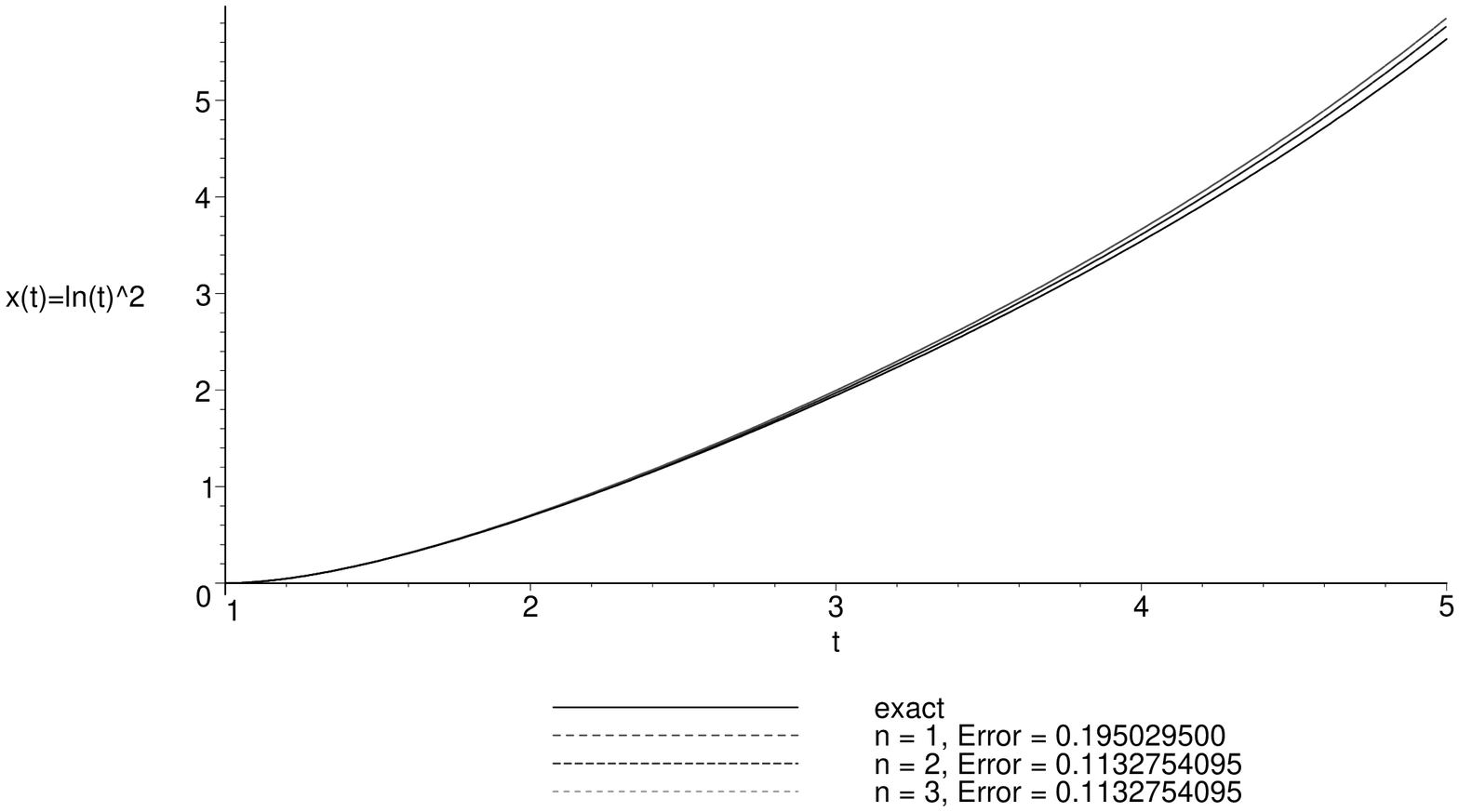}}
\end{center}
\caption{Exact and numerical approximations obtained from Theorem~\ref{MainthmHderivative}
for the left Hadamard derivative of variable fractional order
${_1\mathcal{D}_t^{\a}} x(t)$ with $\a=\frac{t}{10}$
and $x(t)=(\ln t)^2$, $t\in[1,5]$.}
\label{BaN2}
\end{figure}
\begin{figure}[ht]
\begin{center}
\psfrag{x\(t\)=ln\(t\)^2}{\tiny  \hspace*{-0.3cm} ${_1\mathbb{D}_t^{\a}} x(t)$}
\psfrag{t}{\tiny $t$}
\subfigure[$n=1$]{\includegraphics[scale=0.35,angle =-90]{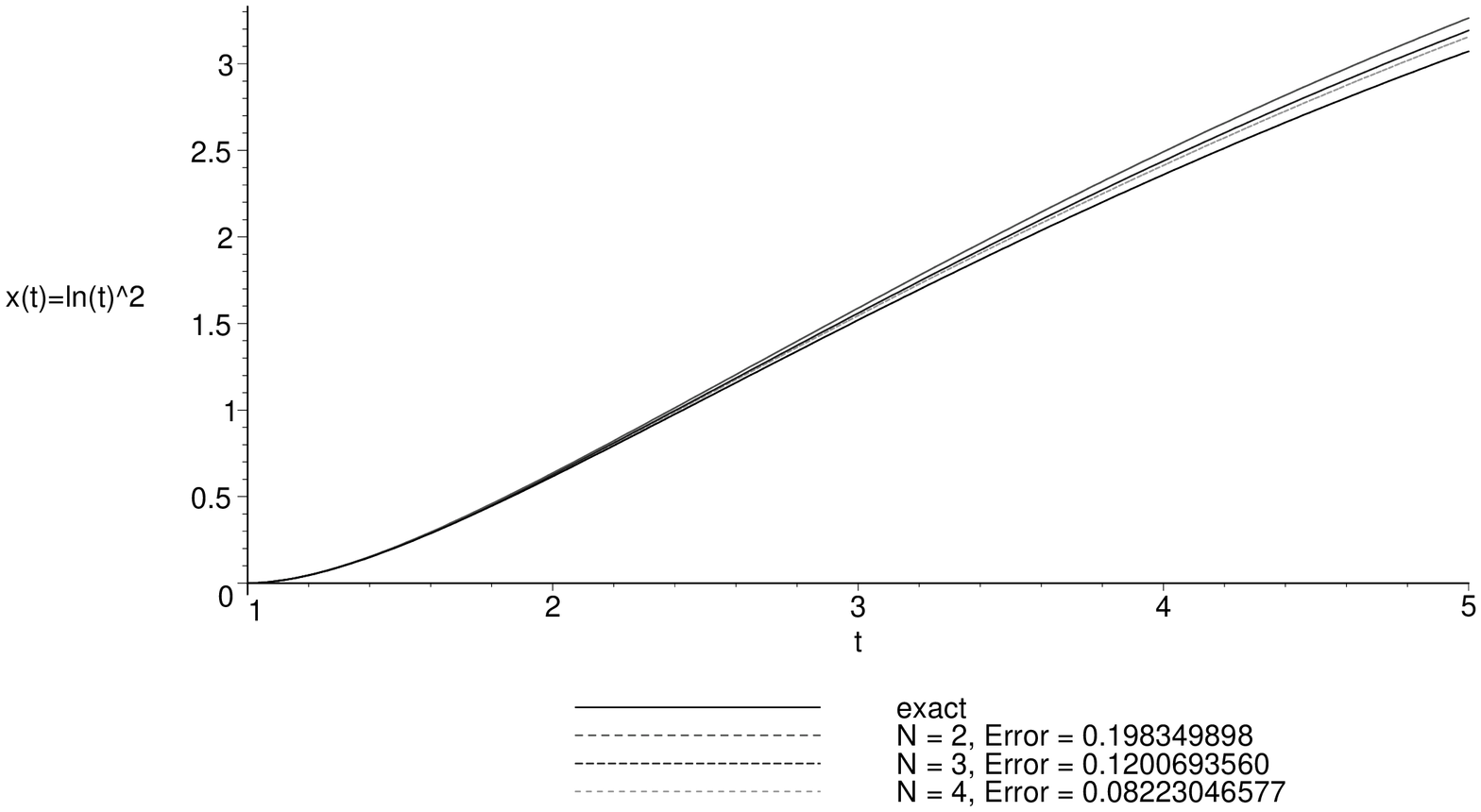}}
\psfrag{x\(t\)=ln\(t\)^2}{\tiny  \hspace*{-0.3cm} ${_1\mathbb{D}_t^{\a}} x(t)$}
\psfrag{t}{\tiny $t$}
\subfigure[$N=4$]{\includegraphics[scale=0.35,angle =-90]{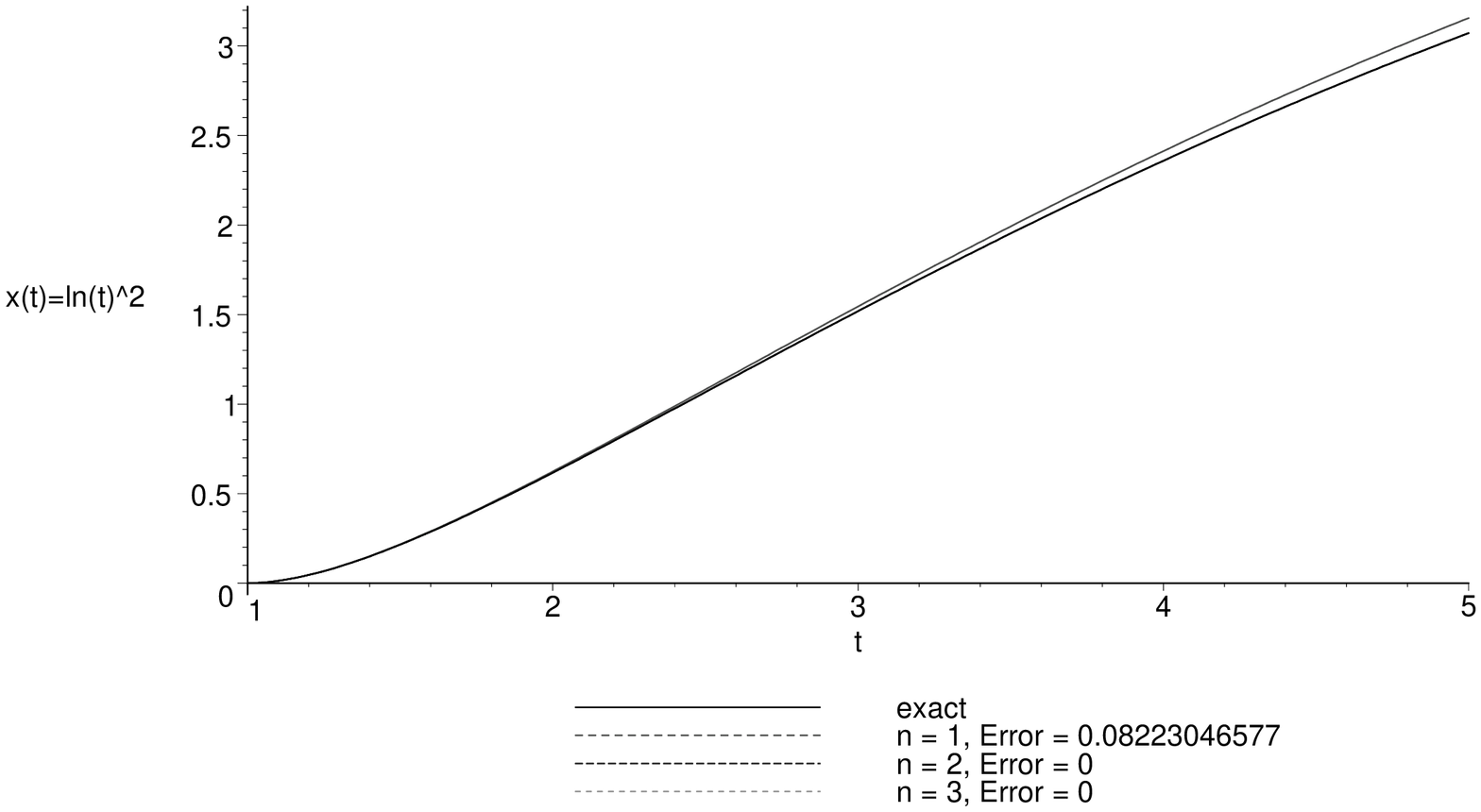}}
\end{center}
\caption{Exact and numerical approximations obtained from Theorem~\ref{MainthmHMderivative}
for the Hadamard--Marchaud derivative of variable fractional order
${_1\mathbb{D}_t^{\a}} x(t)$ with $\a=\frac{t}{10}$,
$x(t)=(\ln t)^2$, $t\in[1,5]$.}
\label{BaN3}
\end{figure}


The formulas given in this paper (see Section~\ref{sec:approx})
can be used in many different kinds of problems.
Since they only depend on integer-order derivatives, we simply replace the fractional
operators that appear in the problem with our expansion formulas. By doing so,
the problem no longer depends on fractional integrals or derivatives, and we obtain
a classical problem with dependence on integer-order derivatives only. Then we may
apply known techniques (analytical or numerical) to solve the problem and thus
obtaining an approximation of the initial problem. We exemplify the procedure
with a fractional differential equation and with a fractional variational problem, depending
on the left Hadamard--Marchaud fractional derivative of order $\DS\a=\frac{t}{10}$,
with $t\in[1,5]$. For other types of fractional operators or different problems,
the procedure is similar. In each case, the result $x$ is compared with the exact
solution $\overline x$ by computing the error $E(x,\overline x)$ with the $L^2$ norm
\eqref{eq:L2norm}.


\subsection{Fractional differential equations}
\label{sec:application}

Consider the following fractional differential equation of variable order:
\begin{equation}
\label{eq:FDE:VO}
\begin{cases}
{_1\mathbb{D}_t^{\a}} x(t)+x(t)=\frac{2}{\Gamma\left(3
-\frac{t}{10}\right)}(\ln t)^{2-\frac{t}{10}}+(\ln t)^2,
\quad t\in[1,5],\\
x(1)=0.
\end{cases}
\end{equation}
It is easy to check that $\overline{x}(t)=(\ln t)^2$ is a solution to \eqref{eq:FDE:VO}.
Since we only have one boundary condition, we must choose an approximation formula using
only the first derivative of $x$ and with size $N\geq 2$, that is,
\begin{multline}
\label{fracapprox}
{_1\mathbb{D}_t^{\a}} x(t)\approx A(0)(\ln t)^{-\a}x(t)
+A(1)(\ln t)^{1-\a}t x'(t)\\
+\sum_{k=2}^N B(k)(\ln t)^{1-k-\a}V_k(t),
\end{multline}
where
\begin{equation*}
\begin{split}
A(k)&=\frac{1}{\Gamma(k+1-\a)}\left[1+\sum_{p=2-k}^N
\frac{\Gamma(p+\a-1)}{\Gamma(\a-k)(p-1+k)!}\right],
\quad k = 0,1,\\
B(k)&=\frac{\Gamma(k+\a-1)}{\Gamma(-\a)\Gamma(1+\a)(k-1)!},
\quad k=2,\ldots,N.
\end{split}
\end{equation*}
Thus, performing this approximation, we obtain the system
\begin{equation}
\label{eq:stODE:syst}
\begin{cases}
\DS\left[A(0)(\ln t)^{-\a}+1\right]x(t)+A(1)(\ln t)^{1-\a}t x'(t)
+\sum_{k=2}^N B(k)(\ln t)^{1-k-\a}V_k(t)\\
\quad =\DS\frac{2}{\Gamma\left(3-\frac{t}{10}\right)}(\ln t)^{2-\frac{t}{10}}+(\ln t)^2,\\
\DS V'_k(t)=(k-1)(\ln t)^{k-2}\frac{x(t)}{t},  \qquad k=2,3,\ldots,N,\\
x(1)=0, \quad V_k(1)=0, \qquad k=2,3,\ldots,N.
\end{cases}
\end{equation}
Now we apply any standard numerical tool to solve \eqref{eq:stODE:syst}, which is simply
a system of $N$ first order differential equations. Such solution is an approximation
to the solution of the variable order fractional problem \eqref{eq:FDE:VO}, and the error will decrease as $N$ goes
to infinity. The result obtained using the routine \textsf{dsolve} of Maple is shown in Figure~\ref{diff}.
\begin{figure}[ht]
\begin{center}
\psfrag{x\(t\)=ln\(t\)^2}{\hspace{0.5cm} \small $x(t)$}
\psfrag{t}{\small $t$}
\includegraphics[scale=0.5, angle=-90]{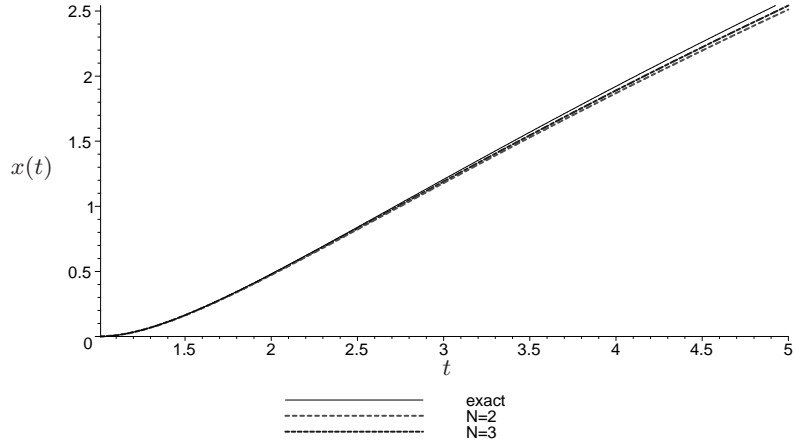}\\
\caption{Analytical and approximate solutions to the
variable fractional order problem \eqref{eq:FDE:VO}.}\label{diff}
\end{center}
\end{figure}


\subsection{Fractional problems of the calculus of variations}
\label{sec:appl:fcv}

The study of fractional problems of the calculus of variations
is a subject of strong current research and practical interest
\cite{FCV:the:book2,FCV:the:book}. For fractional variational
problems of variable order, see \cite{Od1,Od2}.
We now consider a fractional problem of the calculus of variations,
which is described in the following way: minimize the functional
\begin{equation}
\label{eq:FCV}
\mathcal{J}(x)=\int_1^{5}\left( {_1\mathbb{D}_t^{\a}} x(t)
-\frac{2}{\Gamma\left(3-\frac{t}{10}\right)}(\ln t)^{2-\frac{t}{10}}\right)^2\,dt
\end{equation}
subject to the boundary conditions
\begin{equation}
\label{eq:FCV:BC}
x(1)=0 \quad \mbox{and}\quad x(5)=(\ln 5)^2.
\end{equation}
Again, the solution to this problem is $\overline{x}(t)=(\ln t)^2$,
since functional \eqref{eq:FCV} is non-negative, it vanishes when evaluated
at $\overline x$, and the boundary conditions \eqref{eq:FCV:BC} are verified for $\overline x$.
The numerical method is described in the following way. Replacing the fractional derivative
by the approximation \eqref{fracapprox}, we get a new problem,
which is an approximation of the initial one. Define the control
$$
u(t) = A(0)(\ln t)^{-\a}x(t)+A(1)(\ln t)^{1-\a}t x'(t)
+\sum_{k=2}^N B(k)(\ln t)^{1-k-\a}V_k(t).
$$
Thus, we obtain the differential equation
\begin{equation}
\label{dynamic_const}
x'(t)=\frac{(\ln t)^{\a-1}}{A(1)t}u(t)-\frac{A(0)(\ln t)^{-1}}{A(1)t}x(t)
-\sum_{k=2}^N \frac{B(k)}{A(1)t}(\ln t)^{-k}V_k(t).
\end{equation}
Denote the right-hand side of equation \eqref{dynamic_const} by $f\left(t,x,V,u\right)$.
So, the approximated optimal control problem is described in the following way: minimize the functional
$$
\tilde{\mathcal{J}}(x,V,u)=\int_1^{5}\left( u(t)-\frac{2}{\Gamma\left(3
-\frac{t}{10}\right)}(\ln t)^{2-\frac{t}{10}}\right)^2dt,
$$
where $V=(V_2,\ldots, V_n)$, subject to the dynamic constraints
$$
\begin{cases}
\DS x'(t)=f\left(t,x,V,u\right),\\
\DS V'_k(t)=(k-1)(\ln t)^{k-2}\frac{x(t)}{t},
\qquad k=2,3,\ldots,N,
\end{cases}
$$
and the boundary conditions
$$
\begin{cases}
x(1)=0,\\
x(5)=(\ln 5)^2,\\
V_k(1)=0, \qquad k=2,3,\ldots,N.
\end{cases}
$$
To solve it, we apply the necessary optimality conditions
that every solution of this optimal control problem must satisfy.
To this end, consider the Hamiltonian function defined by
\begin{multline*}
H(t,x,V,u,\lambda)=\left( u(t) -\frac{2}{\Gamma\left(3
-\frac{t}{10}\right)}(\ln t)^{2-\frac{t}{10}}\right)^2
+\lambda_1 f\left(t,x,V,u\right)\\
+\sum_{k=2}^N \lambda_k (k-1)(\ln t)^{k-2}\frac{x(t)}{t},
\end{multline*}
where $\lambda=(\lambda_1,\lambda_2,\ldots,\lambda_N)$.
We then apply the Pontryagin maximum principle to this problem
\cite{Pontryagin:book}, which asserts that
$$
\frac{\partial H}{\partial u}=0,
\quad x'=\frac{\partial H}{\partial \lambda_1},
\quad V_k'=\frac{\partial H}{\partial \lambda_k},
\quad \lambda_1'=-\frac{\partial H}{\partial x},
\quad \lambda_k'=-\frac{\partial H}{\partial V_k}.
$$
We obtain the system of differential equations
\begin{equation}
\label{eq:CNP:SH}
\begin{cases}
\DS x'(t)= \frac{2\ln t}{\Gamma\left(3-\frac{t}{10}\right)A(1)t}
-\frac{(\ln t)^{2\a-2}}{2(A(1)t)^2}\lambda_1(t)
-\frac{A(0) (\ln t)^{-1}}{A(1)t}x(t)\\
\qquad \quad \displaystyle -\sum_{k=2}^N \frac{B(k)}{A(1)t}(\ln t)^{-k}V_k(t),\\
\DS V_k'(t)=(k-1)(\ln t)^{k-2}\frac{x(t)}{t}, \quad k=2,\ldots,N,\\
\DS\lambda_1'(t)=\frac{A(0)(\ln t)^{-1}}{A(1)t}\lambda_1
-\sum_{k=2}^N(k-1)\frac{ (\ln t)^{k-2}}{t}\lambda_k(t),\\
\DS\lambda_k'(t)=\frac{B(k)(\ln t)^{-k}}{A(1)t}\lambda_1,
\quad k=2,\ldots,N,
\end{cases}
\end{equation}
subject to the boundary conditions
\begin{equation}
\label{eq:CNP:BC}
\begin{cases}
x(1)=0,\\
x(5)=(\ln 5)^2,\\
V_k(1)=0, \qquad k=2,3,\ldots,N\\
\lambda_k(5)=0,\quad k=2,3,\ldots,N.
\end{cases}
\end{equation}
We solve this system with the command \textsf{dsolve} of Maple,
and the result is depicted in Figure~\ref{cv-maple}.
\begin{figure}[ht]
\begin{center}
\psfrag{x\(t\)=ln\(t\)^2}{\hspace{0.5cm} \small $x(t)$}
\psfrag{t}{\small $t$}
\includegraphics[scale=0.5,angle=-90]{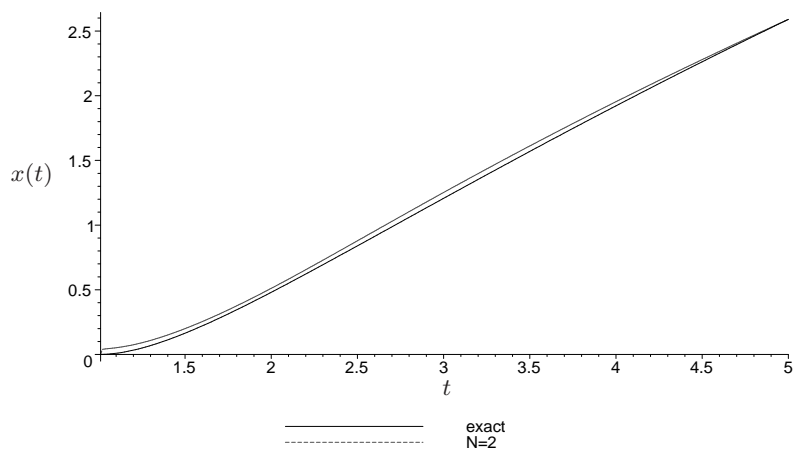}\\
\caption{Exact and approximate solution to the
variable order fractional problem of the calculus of variations
\eqref{eq:FCV}--\eqref{eq:FCV:BC}.}
\label{cv-maple}
\end{center}
\end{figure}


\section*{Acknowledgments}

Work partially supported by Portuguese funds through the
\emph{Center for Research and Development in Mathematics and Applications} (CIDMA),
and \emph{The Portuguese Foundation for Science and Technology} (FCT),
within project PEst-OE/MAT/ UI4106/2014. Torres was also supported
by EU funding under the 7th Framework Programme FP7-PEOPLE-2010-ITN,
grant agreement 264735-SADCO; and by the FCT project PTDC/EEI-AUT/1450/2012,
co-financed by FEDER under POFC-QREN with COMPETE reference FCOMP-01-0124-FEDER-028894.
The authors are grateful to two referees for their valuable
comments and helpful suggestions.



\end{document}